\documentclass[11pt,reqno,oneside]{amsart}
\usepackage{verbatim,amsmath,amssymb,cite,epsfig,xspace,color,subfigure,bbm}
\usepackage[foot]{amsaddr}

\theoremstyle{plain}

\newtheorem{lemma}{Lemma}[section]
\newtheorem{proposition}{Proposition}[section]
\newtheorem{corollary}{Corollary}[section]
\theoremstyle{definition}

\theoremstyle{remark}

\newcommand{\calL}{\mathcal{L}}

\newcommand{\RR}{\mathbb{R}}
\newcommand{\la}{\langle}
\newcommand{\ra}{\rangle}
\newcommand{\abs}[1]{\lvert #1 \rvert}

\newcommand{\qnlbilinear}{b_{\mathsf{qnl}}}
\newcommand{\nlbilinear}{b_{\mathsf{n}}}

\newcommand{\helen}[1]{{#1}}

\title{Quasinonlocal coupling of nonlocal diffusions}

\author{Xingjie Helen Li}
\address{Department of Mathematics and Statistics, University of North
  Carolina at Charlotte, Charlotte NC 28223, USA}
\email{xli47@uncc.edu}

\author{Jianfeng Lu}
\address{Department of Mathematics, Department of
  Physics, and Department of Chemistry, Duke University, Box 90320,
  Durham NC 27708, USA}
\email{jianfeng@math.duke.edu}

\thanks{The work of J.L.~was supported in part by the National Science
  Foundation under awards DMS-1312659 and DMS-1454939.
  The work of X.L.~is supported in part by the Simons Foundation Collaboration Grant
  with award number 426935.
  We thank the helpful discussion with Xiaochuan Tian. 
}

\begin{document}

\maketitle

\begin{abstract}
  We developed a new {\helen{self-adjoint}}, consistent, and stable coupling
  strategy for nonlocal diffusion models, inspired by the
  quasinonlocal atomistic-to-continuum method for crystalline solids.
  The proposed coupling model is coercive with respect to the energy
  norms induced by the nonlocal diffusion kernels as well as the $L^2$
  norm, and it satisfies the maximum principle.  A finite difference
  approximation is used to discretize the coupled system, which
  inherits the property from the continuous formulation. Furthermore,
  we \helen{design a numerical example} which shows the discrepancy
  between the fully nonlocal and fully local diffusions, whereas the
  result of the coupled diffusion agrees with that of the fully
  nonlocal diffusion.
\end{abstract}



\section{Introduction}\label{section_intro}

Nonlocal models have been developed and received lot of attention
\helen{in recent years} to model systems with important scientific and engineering
applications, for example, the phase transition
\cite{Bates1999,Fife2003}, the nonlocal heat conduction
\cite{Bobaru2010b}, the peridynamics model for mechanics
\cite{Silling2000}, just to name a few. These nonlocal models give
rise to new questions and challenges to applied mathematics both in
terms of analysis and numerical algorithms.

While it is established that the nonlocal formulations can often
provide {more accurate descriptions} of the systems, the nonlocality
also increases the computational cost compared to conventional models
based on PDEs. As a result, methodologies that couple nonlocal models
with more localized descriptions have been investigated in recent
years. The goal is to combine the accuracy of nonlocal models with the
computational and modeling efficiency of local PDEs. The basic idea is
to apply the nonlocal model to those parts of the domain that require
the improved accuracy, and use the more efficient local PDE model in
other regions to reduce the computational costs. Besides, from a
modeling point of view, the nonlocal models usually give rise to
modeling challenges near the boundary, as volumetric boundary
conditions are needed and those are often hard to characterize or
match with the physical setup of the system. It is often easier to
switch to the local models near the boundary so that usual boundary
conditions can be used. {\helen{Coupling different nonlocal models is
    also important as it has applications in modeling hierarchically
    structured materials \cite{Gao2007,Bobaru2011a},  nonlocal heat
    conductors \cite{Bobaru2010b,Delia2014a}, etc.  }}

In the past decades, many works {have been} developed {for} numerical analysis of
nonlocal models, for example
\cite{Chasseigne2006a,Du2012a,Du2013a,Du2013b,Kriventsov,Tian2016b},
which {\helen{gave understanding to properties and asymptotic
behaviors}}. In particular, many strategies are proposed to
couple together local-to-nonlocal or two nonlocal models with
different nonlocality. For example, (1) Arlequin type domain
decomposition, see e.g.,
\cite{prudhomme:modelingerrorArlequin,Han2012}; (2) Optimal-control
based coupling, e.g., \cite{Delia2015a,Delia2014a}; (3) Morphing
approach as in \cite{Lub2012a}; (4) Force-based blending mechanism,
see e.g., \cite{Seleson2013a,Seleson2015a}; (5) Energy-based blending mechanism, see e.g., \cite{Silling2015a,Tian2016a}, just to name a few.

{\helen{ Nonlocal models are considered as top-down continuum
    approaches which use integral formulations to represent nonlocal
    spatial interactions
    \cite{Petersson1976,Kohlhoff1989,Silling2000}.}  \helen{Within a
    nonlocal model, each material point interacts through short-range
    forces with other points inside a horizon of prescribed radius
    $\delta$, which leads to a nonlocal integral-type continuum theory
    \cite{Silling2000}. Complementary to top-down nonlocal continuum
    approaches are bottom-up atomistic-to-continuum (AtC) approaches,
    which give atomistic accuracy near defects such as crack tips and
    continuum finite element efficiency elsewhere. Notice that
    top-down nonlocal models and bottom-up AtC models actually share a
    lot interconnections (see for example,
    \cite{bely1999,Bobaru2010a,Bobaru2011a,Lipton2016a} for more
    discussions).  }} Many of the coupling strategies developed for
nonlocal models were also inspired or connected to those of AtC
coupling methods for crystalline materials (see e.g., the review
articles \cite{Luskin2013a,Tadmor2009a}). The AtC coupling can be
indeed understood as a nonlocal-local coupling as well, since the
atomistic models typically involve interactions of atoms
{\helen{within the interaction range (i.e., horizon) which are}}
beyond the nearest neighbor. {\helen{ Additionally, the nonlocal
    models and AtC models have some similarities in one-dimension
    after finite difference discretization.}}

In fact, the coupling strategy we proposed in this work also borrows
the idea of quasinonlocal coupling \cite{Shimokawa:2004,E:2006} {\helen{
which was developed in the context of
AtC method, and a detailed interpretation of quasinonlocal coupling will be given in section~\ref{section_qnl}.}} \helen{The proposed method gives} a {self-adjoint}
coupling kernel in the divergence form using the terminology of
\cite{CaffarelliChanVasseur:11}. In physics terms, the coupling
naturally satisfies the Newton's third law since the coupling is done
on the energy level. Moreover, the coupling we proposed satisfies the
patch-test consistency, $L^2$ stability, and the maximum principle. As far
as we know, none of existing coupling strategies {satisfies} all these
properties.  We believe the coupling strategy is more advantageous
over the existing methods.

The paper is organized as follows. In
section~\ref{section_full_nonlocal}, we give a brief review of
nonlocal diffusion problem.  In section~\ref{section_qnl}, we propose
the quasinonlocal coupling, and prove it is {\helen{self-adjoint}} and patch-test
consistent. In addition, we prove that the quasinonlocal diffusion is
positive-definite with respect to the energy norm induced by the
nonlocal diffusion {kernels} as well as the $L^2$ norm, and it
satisfies the maximum principle.  In section~\ref{section_fdm}, we
provide a first order finite difference approximation for the purpose
of numerical implementation, which keeps the properties of
continuous level. In section~\ref{section_num}, we verify our
theoretical results by a few numerical examples.

\section{The nonlocal diffusion}\label{section_full_nonlocal}

We first review the definition and properties of the nonlocal
diffusion problem, following mainly the notations in
\cite{Du2012a,Du2013a}. Consider an open domain $\Omega
\subset\RR^d$ and for $u(x): \, \RR^d \to \RR$,
the (linear) nonlocal diffusion operator $\calL_{\delta}$ is defined
as
\begin{equation}\label{nonlocal_operator}
\calL_{\delta} u(x):= \int_{\RR^d}\left(u(y)-u(x)\right)\gamma_{\delta}(x,y)dy,\quad \forall x\in \Omega,
\end{equation}
where $\delta$ is the horizon parameter and
$\gamma_{\delta}(x,y): \RR^d\times\RR^d\rightarrow \RR$
{\helen{denotes a nonnegative symmetric function}}.  {\helen{Under the
    assumption of translational invariance and isotropy, the kernel
    $\gamma_{\delta}(x,y)$ reduces to a radial function depending on
    the distance $|x-y|$ and }} is given by{\helen{\cite{Lehoucq2010}
  }}
\begin{equation}\label{nonlocal_kernel}
  \gamma_{\delta}(x,y)=\gamma_{\delta}(\helen{\abs{x - y}})=\frac{1}{\delta^{d+2}} \gamma\left(\frac{\helen{\abs{x - y}}}{\delta}\right),
\end{equation}
where $d$ is the spatial dimension and $\gamma$ is a non-negative
radially symmetric nonlocal diffusion kernel which satisfies
\begin{itemize}
\item \helen{Translational invariance and isotropy:}
 $\gamma(x,y)=\gamma(\abs{y-x})\geq 0$;
\item Compact support: $\gamma(x, y) = 0 $ if $\abs{x - y} \geq 1$;
\item Finite second moment: $\int s^2 \gamma(s)ds<\infty$. Note that
  due to the scaling choice in \eqref{nonlocal_kernel}, the second
  moment is scale invariant, \textit{i.e.}, $\int s^2
  \gamma_{\delta}(s) ds$ takes the same value for any $\delta$.
\end{itemize}
With Dirichlet boundary condition, the initial-boundary value problem
for the nonlocal diffusion is then given by:
\begin{equation}\label{nonlocal_divergence_form}
\begin{cases}
\displaystyle\frac{\partial u_n}{\partial t}=\calL_{\delta} u_n(x)= \int_{B_{\delta}(x)}\left(u_n(y)-u_n(x)\right)\gamma_{\delta}(x,y)dy, \quad \forall x\in \Omega,\, \forall t>0\\
u_n(x,t)=0,\quad \forall x\in \Omega_{\mathcal{I}},\, \forall t\ge 0,\\
u_n(x,0)= u_n^0(x), \quad \forall x \in \Omega,
\end{cases}
\end{equation}
where the subscript in $u_n$ stands for ``nonlocal'' and
$\Omega_{\mathcal{I}} = \RR^d \backslash \Omega$ is known as the
interaction domain \cite{Du2013b}, {\helen{on which}} the volumetric Dirichlet
boundary condition is imposed on $u_n$. Since the kernel
$\gamma_{\delta}(x, y)$ is zero if $\abs{x - y} > \delta$, we have
restricted the integration \eqref{nonlocal_divergence_form} in
$B_{\delta}(x)$: the $\delta$-ball around $x$ and also the interaction
domain $\Omega_{\mathcal{I}}$ can be also restricted to the
$\delta$-neighborhood of $\Omega$:
{\helen{
\begin{equation}\label{def_OmegaI}
\Omega_{\mathcal{I}}:=\{y\in \mathbb{R}^d\setminus \Omega: |y-x|{\le}\delta \text{ for some } x\in \Omega\}.
\end{equation}
}}
Other boundary conditions can be used, for example if $\Omega$ is a
box region $\Omega = [0, L)^d$, we can impose the periodic boundary
conditions on $u_n$. The corresponding initial-boundary value problem
can be written down analogously.

The nonlocal diffusion operator is associated with the
Hilbert spaces given by
\[
S_{\delta}:=\left\{ u\in L^2(\Omega\cup \Omega_{\mathcal{I}}): \int_{
    \Omega\cup \Omega_{\mathcal{I}} }\int_{\Omega\cup
    \Omega_{\mathcal{I}} }
  \gamma_{\delta}(x,y)\left(u(y)-u(x)\right)^2 dxdy <\infty, \;
  u\big|_{\Omega_{\mathcal{I}}}=0  \right\}.
\]
The induced nonlocal energy norm is denoted as
$\|\cdot\|_{S_{\delta}}$:
{\helen{
\begin{equation}\label{def_norm}
\|u\|_{S_{\delta}}^2:= \int_{
    \Omega\cup \Omega_{\mathcal{I}} }\int_{\Omega\cup
    \Omega_{\mathcal{I}} }
  \gamma_{\delta}(x,y)\left(u(y)-u(x)\right)^2 dxdy,\quad \forall u\in S_{\delta}.
\end{equation}
}}
The properties of the nonlocal kernel as well as the nonlocal energy
norms are investigated and discussed in many recent works, we refer
the readers to \cite{Du2013a,Tian2013a,Tian2016b,Delia2014a} and
references therein. We remark {\helen{in particular that}} the nonlocal energy norm
satisfies the nonlocal Poincar\'{e} inequality
\cite[Eq. (2.11)]{Delia2014a}, which will be used in our analysis in
the sequel:
\begin{equation}\label{nonlocal_poincare}
\|u\|_{L^2(\Omega \cup \Omega_{\mathcal{I}})} \le
C_{d, \delta}\|u\|_{S_{\delta}}.
\end{equation}

The horizon parameter should be chosen according the physical property
of the underlying system. We are interested in the cases where the
horizon parameter changes across the domain of interest. In this case,
we shall couple two nonlocal diffusion operators with different
horizon {parameters} together. In the next section, we propose {\helen{a way of coupling based on the quasinonlocal idea.}} \helen{While we focus in this work the coupling of two nonlocal diffusions, we note that the idea can be extended to coupling of local and nonlocal diffusions, which will be considered in future works.}

\section{The quasinonlocal coupling}\label{section_qnl}

In this section, we propose a coupling scheme for multiscale nonlocal
diffusions, and prove the consistency, $L^2$ stability and the maximum
principle for the new coupling operator.

To better convey the idea, we will limit our discussion to one
dimensional case in this work. For simplicity of notation, we assume
that $\Omega = [-a, b]$ and it is divided into two parts $\Omega_1 =
[-a, 0]$ and $\Omega_2 = [0, b]$ with the interface at $0$. We assume
that within $\Omega_1$, the nonlocal diffusion kernel
$\gamma_{\delta_1}$ should be employed and $\gamma_{\delta_2}$ should
be used in $\Omega_2$. \helen{Without loss of generality, we assume $\delta_1 > \delta_2$ and thus
$\Omega_1$ is a more nonlocal region compared to $\Omega_2$. We further} assume that
$\delta_1=M\delta_2$ with $M\in \mathbb{N}$ being an integer. \helen{Our coupling strategy requires $M$ to be an integer; it might be interesting to study how to extend to arbitrary ratio of $\delta_1$ and $\delta_2$.}

To ensure the symmetry, we will derive the coupled nonlocal diffusion operator (i.e., negative of the force) from \helen{a} total energy. Recall that
{\helen{the total energy associated with the
kernel}} $\gamma_{\delta}$ is
\begin{equation}\label{total_nonlocal_energy}
E^{{\mathsf{tot}}, \delta}(u)
=\frac{1}{4}\int_{ \mathbb{R}}
\int_{\mathbb{R} } \gamma_{\delta}(\abs{y-x}) \left(u(y)-u(x)\right)^2~dxdy.
\end{equation}
{\helen{An intuitive coupling idea}} is to combine the {energies} associated with
$\delta_1$ and $\delta_2$, respectively. For instance, we use $\gamma_{\delta_2}$ if
$x, y \in \Omega_2$ and $\gamma_{\delta_1}$ otherwise. It can be
verified though, such coupling strategy does not satisfy the
patch-test consistency. The resulting operator (as {\helen{the first variation of the energy}}) $\mathcal{L}$ does not {\helen{annihilate affine
functions}},
which is of course a property one would hope a diffusion operator
should satisfy.

\subsection{Quasinonlocal coupling with geometric reconstruction}

To overcome the difficulty of naive coupling strategies, our
construction does not simply try to vary $\delta$ in
\eqref{total_nonlocal_energy} across the domain, but instead,
\helen{we follow} the geometric reconstruction reformulation for the
quasinonlocal method \cite{E:2006}. \helen{Let us first present the
  formulation of the coupling before explaining the ideas behind.}
{\helen{Our proposed total energy of the quasinonlocal coupling is
    given by
\begin{align}\label{total_coupled_energy}
E^{{\mathsf{tot}}, {\mathsf{qnl}}}(u)
=&\frac{1}{4}\int_{x,y \in \mathbb{R}, x\le 0 \text{ or } y\le 0}
 \gamma_{\delta_1}(\abs{y-x}) \left(u(y)-u(x)\right)^2~dx~dy\\
&\; +\frac{1}{4}\int_{x,y \in \mathbb{R}, x>0 \text{ and }  y>0} \gamma_{\delta_1}(\abs{y-x}) \nonumber\\
&\;\;
\cdot \frac{1}{M} \sum_{j=0}^{M-1} \left(u(x+\frac{j+1}{M}(y-x))-u(x+\frac{j}{M}(y-x))\right)^2\,
M^2 ~dx~dy.\nonumber
\end{align}
}} {\helen{ Note that the energy functional only uses one interaction
  kernel $\gamma_{\delta_1}$ (the more non-local one) throughout the
  entire domain $\Omega=\Omega_1\cup\Omega_2$.  In the subregion
  $\Omega_2$, instead of changing to the kernel $\gamma_{\delta_2}$,
  we change the difference $(u(y) - u(x))^2$ to an averaged
  version. This is coined
  as ``geometric reconstruction'', since it reconstructs $u(y) - u(x)$
  by differences of $u$ evaluated at points that are at most {of distance}
  $\delta_2$ to each other. Therefore, while a longer range kernel
  $\gamma_{\delta_1}$ is used in $\Omega_2$, the energy effectively
  still only involves interaction no further than $\delta_2$ distance.
}}
More concretely, to
link the {regions from} kernel $\gamma_{\delta_1}$ to $\gamma_{\delta_2}$
with $\delta_1 = M \delta_2$, in the local region $\Omega_2$, we
replace $ {\helen{\gamma_{\delta_2}}}(\abs{y-x})\left( u(y)-u(x)\right)^2$ by (see Figure~\ref{Figure:1D_qnl_demo} for an illustration)
\begin{multline*}
 \gamma_{\delta_1}(\abs{y-x}) \frac{1}{M} \sum_{j=0}^{M-1} \left(u\bigl(x+\frac{j+1}{M}(y-x)\bigr)-u\bigl(x+\frac{j}{M}(y-x)\bigr)\right)^2
\left(\frac{\delta_1}{\delta_2}\right)^2\\
= \gamma_{\delta_1}(\abs{y-x}) \frac{1}{M} \sum_{j=0}^{M-1}\left(u\bigl(x+\frac{j+1}{M}(y-x)\bigr)-u\bigl(x+\frac{j}{M}(y-x)\bigr)\right)^2\,
M^2.
\end{multline*}
Note that the kernel remains intact, and we adopt the \emph{geometric
  reconstruction} to replace
\begin{equation*}
  u(y) - u(x) \rightarrow \left(u \bigl(x + \frac{j+1}{M} (y - x)\bigr) -
    u \bigl(x + \frac{j}{M} (y - x)\bigr) \right) M
\end{equation*}
for $j = 0, \ldots, M-1$ \helen{and average the modulus square of the
  result over the $M$ possibilities}. Note that if
$\abs{x - y} \leq \delta_1$, the difference on the right is $u$
evaluated at points with distance at most
$\frac{\delta_1}{M} = \delta_2$; thus effectively we reconstruct the
difference $u(y) - u(x)$ by a more local interaction (and hence the
idea was referred as geometric reconstruction scheme in
\cite{E:2006}).  In fact, if we adopt such reconstruction everywhere
in the computational domain, we will get the nonlocal diffusion with
the kernel $\gamma_{\delta_2}$, as shown in the following Proposition.

\begin{figure}[htp]
\centering
\includegraphics[height=3cm, width=7 cm]{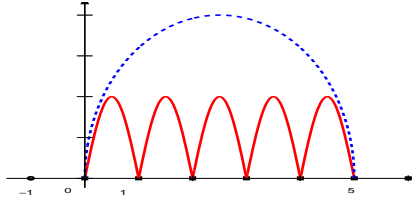}
\caption{An example of quasinonlocal construction in one dimension:
  $x\le 0$ is the region with kernel $\gamma_{\delta_1}$ and $x> 0$ is
  the region with kernel $\gamma_{\delta_2}$ ($M = 5$). The blue dash
  bond representing the difference $u(y) - u(x)$ is geometrically
  reconstructed by the red solid bonds with length at most
  $\delta_2$. }\label{Figure:1D_qnl_demo}
\end{figure}

\begin{proposition}\label{prop_approx_energy_consist}
  The energy functional defined on the entire domain
  $\Omega\cup\Omega_{\mathcal{I}}$ with geometric reconstruction
  \begin{align}
    E^{{\mathsf{tot}}, {\mathsf{gr}}}(u)
    := \frac{1}{4}&\int_{x,y \in \Omega\cup\Omega_{\mathcal{I}}}
    \gamma_{\delta_1}(\abs{y-x}) \label{total_approx_energy} \\
    & \frac{1}{M} \sum_{j=0}^{M-1}
    \left(u\left(x+\frac{j+1}{M}(y-x)\right)-
      u\left(x+\frac{j}{M}(y-x)\right)\right)^2 M^2 ~dxdy\nonumber
\end{align}
is equivalent to the total
nonlocal energy with diffusion kernel $\gamma_{\delta_2}$:
\begin{align*}
E^{{\mathsf{tot}}, \delta_2}(u):=& \frac{1}{4} \int_{x, y\in\Omega\cup\Omega_{\mathcal{I}}} \gamma_{\delta_2}(\abs{y-x})
\left(u(y)-u(x)\right)^2~dxdy.
\end{align*}
\end{proposition}

\begin{proof}
  Outside the support, the diffusion kernel is zero, {and because of the zero volumetric Dirichlet boundary condition}, so the total
  energy \eqref{total_approx_energy} can be recast as an integral of
  entire real line $\RR$:
 \begin{align}\label{energy_consist_eq1}
 E^{{\mathsf{tot}}, {\mathsf{gr}}}(u)
&\, = \frac{1}{4}\int_{x,y \in \RR }M \gamma_{\delta_1}(\abs{y-x})\nonumber\\
&\sum_{j=0}^{M-1} \left(u\left(x+\frac{j+1}{M}(y-x)\right)-u\left(x+\frac{j}{M}(y-x)\right)\right)^2
~dxdy.
 \end{align}
For each fixed $0\le j\le M-1$ in \eqref{energy_consist_eq1}, we introduce change of variables
\[
z:=x+\frac{j+1}{M}(y-x)=\left(1-\frac{j+1}{M}\right)x+\frac{j+1}{M}y
\]
to replace $y$, thus we obtain
\begin{equation*}
\begin{split}
\frac{1}{4} &\int_{x,y \in \RR } M \gamma_{\delta_1}(\abs{y-x})
{\helen{\left(u\left(x+\frac{j+1}{M}(y-x)\right)-u\left(x+\frac{j}{M}(y-x)\right)\right)^2}}
~dxdy\\
=&\frac{1}{4} \int_{x,z \in \RR }\gamma_{\delta_1} \left(\frac{M}{j+1}\abs{z-x}\right)\frac{M^2}{j+1}
\left(u\left(z\right)-u\left(z-\frac{1}{j+1}(z-x)\right)\right)^2~dzdx\\
=&\frac{1}{4} \int_{x,z \in \RR } \frac{1}{M^3}\gamma_{\delta_2} \left(\frac{1}{j+1}\abs{z-x}\right)\frac{M^2}{j+1}
\left(u\left(z\right)-u\left(z-\frac{1}{j+1}(z-x)\right)\right)^2~dzdx.
\end{split}
\end{equation*}
Next, let $w=z-\frac{1}{j+1}(z-x)=z+\frac{1}{j+1}(x-z)$ to replace $x$, we further get
\begin{equation*}
\begin{split}
\frac{1}{4} \int_{x,z \in \RR } & \frac{1}{M^3}\gamma_{\delta_2} \left(\frac{1}{j+1}\abs{z-x}\right)\frac{M^2}{j+1}
\left(u\left(z\right)-u\left(z-\frac{1}{j+1}(z-x)\right)\right)^2~dzdx\\
=&\frac{1}{4} \int_{w,z \in \RR } \frac{1}{M}\gamma_{\delta_2} \left(\abs{w-z}\right)
\left(u\left(z\right)-u\left(w\right)\right)^2~dzdw
=\frac{1}{M}E^{{\mathsf{tot}}, \delta_2}(u).
\end{split}
\end{equation*}
Summing up all $0\le j\le M-1$, we proved the proposition.
\end{proof}


{\helen{Therefore, }}to couple diffusion kernels $\gamma_{\delta_1}$ and
$\gamma_{\delta_2}$ with interface at $x=0$, we construct the total
coupled energy by using the geometric reconstruction when $x, y \in
\Omega_2$ and {\helen{obtain \eqref{total_coupled_energy}, recalled here}}
\begin{equation*}
\begin{split}
E^{{\mathsf{tot}}, {\mathsf{qnl}}}(u)
=&\frac{1}{4}\int_{x,y \in \mathbb{R}, x\le 0 \text{ or } y\le 0}
 \gamma_{\delta_1}(\abs{y-x}) \left(u(y)-u(x)\right)^2~dx~dy\\
&\quad +\frac{1}{4}\int_{x,y \in \mathbb{R}, x>0 \text{ and }  y>0} \gamma_{\delta_1}(\abs{y-x}) \\ &\qquad
\cdot \frac{1}{M} \sum_{j=0}^{M-1} \left(u(x+\frac{j+1}{M}(y-x))-u(x+\frac{j}{M}(y-x))\right)^2\,
M^2 ~dx~dy.
\end{split}
\end{equation*}
Clearly, when $\delta_1=\delta_2$, we have $M=1$ and thus $E^{{\mathsf{tot}}, {\mathsf{qnl}}}(u)=E^{{\mathsf{tot}}, \delta_1}(u).$

The first variation of the above functional \eqref{total_coupled_energy}
then gives us a coupled diffusion operator, for which we call the
quasinonlocal diffusion, denoted by $\calL^{{\mathsf{qnl}}}$. After
straightforward but somewhat lengthy calculation, which we defer to
the appendix, we obtain that
\begin{equation}\label{def_Lqnl}
\begin{split}
 & \calL^{{\mathsf{qnl}}} u_{\mathsf{qnl}}(x) := - \frac{\delta E^{\mathsf{tot}, \mathsf{qnl} }}{\delta u} u_{\mathsf{qnl}}(x) \\
  &=\begin{cases}
    \displaystyle
    \int_{\RR}\left(u_{ \mathsf{qnl}}(y)-u_{\mathsf{qnl}}(x)\right)\gamma_{\delta_1}(\abs{y-x})dy,  &\text{if }x\le 0, \\ 
    \displaystyle
    \int_{x-\delta_1<y<0}\left(u_{ \mathsf{qnl}}(y)-u_{\mathsf{qnl}}(x)\right)\gamma_{\delta_1}(\abs{y-x})dy \\
    \displaystyle
    \;+ {{ \frac{1}{M} }} \sum_{j=1}^{M-1} \left[\int_{(x-\frac{1}{j}x) <y<  (x+\frac{1}{j}x)}  \left(u_{\mathsf{qnl}}(y)-u_{\mathsf{qnl}}(x)\right)\gamma_{\delta_2}(\abs{y-x})\right]dy\\
    \displaystyle
    \quad\; + \frac{1}{M} \int_{(x-\frac{1}{M}x)< y}
    \left(u_{\mathsf{qnl}}(y)-u_{\mathsf{qnl}}(x)\right)\gamma_{\delta_2}(\abs{y-x})dy,
    &\text{if }0 < x< \delta_1,\\
    \displaystyle
    \int_{\RR}\left(u_{
        \mathsf{qnl}}(y)-u_{\mathsf{qnl}}(x)\right)\gamma_{\delta_2}(\abs{y-x})dy,
   & \text{if }x\ge \delta_1.
\end{cases}
\end{split}
\end{equation}
Note that in the region $\{x \le 0\}$, the diffusion is just the
nonlocal diffusion with kernel $\gamma_{\delta_1}$, while the
diffusion kernel is given by $\gamma_{\delta_2}$ in the region $\{x \ge
\delta_1\}$. The region in between $\{0 < x < \delta_1\}$ can be
viewed as a buffer region that connects the two nonlocal diffusion
operators.  We emphasize that the coupled diffusion operator
$\calL^{{\mathsf{qnl}}}$ is {\helen{self-adjoint}}, as it is derived  from
the coupled energy \eqref{total_coupled_energy} and can be easily
checked \helen{directly} from the definition; and hence the operator
$\calL^{{\mathsf{qnl}}}$ is of divergence form (following the
definition in \cite{CaffarelliChanVasseur:11}). This is the main
motivation behind our construction of the kernel.

The idea of the coupling strategy here is in fact borrowed from the
quasinonlocal coupling in the context of atomistic-to-continuum method
for crystalline solids, first proposed by \cite{Shimokawa:2004},
generalized and analyzed in
\cite{E:2006,MingYang,Shapeev2012a,luskin2011, OrtnerZhang}, in
particular the geometric reconstruction point of view \cite{E:2006}. Hence we {adopt} the name of the quasinonlocal
coupling strategy. 
Actually, as only pairwise interaction is involved in the current
context of nonlocal diffusion operators, the coupling is considerably
easier {than the atomistic-to-continuum coupling for crystals}. In
particular, the extension of atomistic-to-continuum quasinonlocal
coupling to higher dimension for general long range potential is still
an open challenge, despite progresses in \cite{E:2006, Shapeev2012a,
  OrtnerZhang}. While for nonlocal diffusion, we expect the extension
should not pose serious difficulties and will be considered in future
works.

The initial-boundary value problem for the quasinonlocal coupling is given as below:
\begin{equation}\label{QNL_divergence_form}
\begin{cases}
\displaystyle \frac{\partial u_{\mathsf{qnl}}}{\partial t}=\calL^{{\mathsf{qnl}}} u_{\mathsf{qnl}}(x), \quad \forall x \in \Omega,\\
u_{\mathsf{qnl}}(x,t)=0,\quad \forall x\in \Omega_{\mathcal{I}},\quad \forall t\ge 0, \\
u_{\mathsf{qnl}}(x,0)= u_{\mathsf{qnl}}^0(x), \text{ on } \Omega.
\end{cases}
\end{equation}
We now demonstrate the properties of {\helen{the quasinonlocal diffusion operator
$\calL^{{\mathsf{qnl}}}$}} in the following subsections.
\subsection{Consistency}
\begin{lemma} The quasinonlocal diffusion operator
  $\calL^{{\mathsf{qnl}}}$ defined in \eqref{def_Lqnl} is a {\helen{self-adjoint}} operator,
  and is consistent in the sense that $\calL^{\mathsf{qnl}} u = 0$ for
  any affine function $u$.
\end{lemma}
\begin{proof}
  The \helen{self-adjointness has been shown above already.} To check the consistency, consider an
  affine function $u(x)=Fx+u_0$ with both $F$ and $u_0$ {being}
  constants.  Thus, we have
  \[
  u(y)-u(x)=F(y-x).
  \]
  Plugging this into the definition of $\calL^{\mathsf{qnl}}$
  \eqref{def_Lqnl}, we have the following three cases.
\begin{enumerate}
\item[Case I:] $x\le 0$, then
\[
\calL^{\mathsf{qnl}}u(x)=
\int_{\RR}F\left(y-x\right)\gamma_{\delta_1}(\abs{y-x})dy
= \int_{\RR} F s  \gamma_{\delta_1}(\abs{s}) ds = 0,
\]
which comes from the symmetry of the kernel.
\item[Case II:] $0< x< \delta_1$, then
\begin{equation}\label{consist_eq1}
\begin{split}
\calL^{\mathsf{qnl}}u(x)=&
\int_{x-\delta_1<y<0}F\left(y-x\right)\gamma_{\delta_1}(\abs{y-x})dy \\
& \quad +\frac{1}{M} \sum_{j=1}^{M-1} \left[\int_{(x-\frac{1}{j}x) <y<  (x+\frac{1}{j}x)}  F\left(y-x\right)\gamma_{\delta_2}(\abs{y-x})\right]dy\\
&\qquad \qquad \qquad + \int_{(x-\frac{1}{M}x)< y} \frac{1}{M} F\left(y-x\right)\gamma_{\delta_2}(\abs{y-x})dy.
\end{split}
\end{equation}
Because of the symmetry of integral domain and the symmetry of the kernel $\gamma_{\delta_2}$, the second term in \eqref{consist_eq1} equals zero. Thus, \eqref{consist_eq1} becomes
\begin{equation}\label{consist_eq2}
\begin{split}
&\calL^{\mathsf{qnl}}u(x)\\
&=
\int_{x-\delta_1<y<0}F\left(y-x\right)\gamma_{\delta_1}(\abs{y-x})dy\\
&\qquad\quad + \frac{1}{M} \int_{(x-\frac{1}{M}x)< y} F\left(y-x\right)\gamma_{\delta_2}(\abs{y-x})dy\\
 &=\int_{x-\delta_1}^{0} F\left(y-x\right)\gamma_{\delta_1}(\abs{x-y})dy\\
 &\qquad\quad+\frac{1}{M}\int_{(x-\frac{1}{M}x)< y}  F\left(y-x\right) M^3 \gamma_{\delta_1}\left(\abs{M(y-x)}\right)dy\\
 &= \int_{-\delta_1}^{-x}  F s \gamma_{\delta_1}(\abs{s})ds
 + \int_{-x}^{\infty }  F \hat{s}  \gamma_{\delta_1}(\abs{\hat{s}})d\hat{s}
 =\int_{-\delta_1}^{\delta_1}F s \gamma_{\delta_1}(\abs{s})ds=0,
\end{split}
\end{equation}
where we introduce the change of variables $s:=y-x$ and $\hat{s}:=M(y-x)$.

\item[Case III:] $x\ge\delta_1$, then
\[
\calL^{\mathsf{qnl}}u(x)=
\int_{\RR}F\left(y-x\right)\gamma_{\delta_2}(\abs{y-x})dy = \int_{\RR}
F s \gamma_{\delta_2}(\abs{s}) d s = 0,
\]
which again comes from the symmetry of the kernel.
\end{enumerate}

Summarizing all the three cases, we thus obtain the patch-test consistency of the coupled diffusion operator $\calL^{\mathsf{qnl}}$.
\end{proof}

\subsection{Stability analysis of QNL coupling}\label{subsection_stab}

Let us now consider  the stability of the QNL coupling
\eqref{QNL_divergence_form} and \eqref{def_Lqnl}.  As the coupled
diffusion operator $\calL^{\mathsf{qnl}}$ is derived from the total
energy \eqref{total_coupled_energy}, the bilinear form of the QNL
coupling operator is simply given by
\begin{equation}\label{QNL_bilinear}
\begin{aligned}
  b_{\mathsf{qnl}}(u,v)
  = & \int_{x,y \in \mathbb{R}, x\le 0 \text{ or } y\le 0}
 \gamma_{\delta_1}(\abs{y-x}) \left(u(y)-u(x)\right)\left(v(y)-v(x)\right) dxdy\\
& + \int_{x,y \in \mathbb{R}, x>0 \text{ and }  y>0} M \gamma_{\delta_1}(\abs{y-x})
\\
& \qquad\qquad\sum_{j=0}^{M-1} \left(u(x+\frac{j+1}{M}(y-x))-u(x+\frac{j}{M}(y-x))\right) \\
&\qquad\qquad \qquad \qquad\cdot \left(v(x+\frac{j+1}{M}(y-x))-v(x+\frac{j}{M}(y-x))\right) dxdy.
\end{aligned}
\end{equation}
The induced \helen{inner product space and norm}
associated with $\calL^{{\mathsf{qnl}}}$ are
\begin{equation}
  \begin{aligned}\label{qnl_space}
& S_{\mathsf{qnl}}
:=\left\{
u\in L^2(\Omega\cup \Omega_{\mathcal{I}}): b_{\mathsf{qnl}}(u,u) <\infty, \quad u\big|_{\Omega_{\mathcal{I}}}=0
\right\};\,\\
& {\helen{\|u\|_{S_{\mathsf{qnl}}}^2:}}= b_{\mathsf{qnl}}(u,u), \quad \forall u\in S_{\mathsf{qnl}}.
\end{aligned}
\end{equation}
Also recall the bilinear form of the nonlocal kernel
$\gamma_{\delta_1}$:
\begin{equation}\label{nonlocal_bilinear}
\nlbilinear(u,v)
= \int_{\Omega\cup \Omega_{\mathcal{I}}}dx\int_{\Omega\cup \Omega_{\mathcal{I}}} \gamma_{\delta_1}(x,y)\left(u(y)-u(x)\right)\left(v(y)-v(x)\right)dy.
\end{equation}


\begin{proposition}[Stability] \label{lemma_L2_Nonlocal_vs_QNL}
  For the nonlocal kernels
  $\gamma_{\delta_1}(s)=\frac{1}{\delta_1^{3}}
  \gamma(\frac{s}{\delta_1})$ and
  $\gamma_{\delta_2}(s)=\frac{1}{\delta_2^{3}}
  \gamma(\frac{s}{\delta_2})$ with $\delta_1=M\delta_2$ for $M$ {being} a
  given integer, where $\gamma(s)$ is a symmetric scaleless decreasing
  kernel supported on $[0,1]$, we have
  \begin{equation}\label{bilinear_QNL_vs_nonlocal}
    \qnlbilinear(u,u)
    \ge \nlbilinear(u,u) =
    \|u\|_{S_{\delta_1}}^2.
  \end{equation}
\end{proposition}
\begin{proof}
From the definition of the bilinear form \eqref{QNL_bilinear}, we have
\begin{align}\label{1D_stab_eq1}
  b_{\mathsf{qnl}}(u,u)
  = & \int_{x,y \in \Omega\cup \Omega_{\mathcal{I}}, x\le 0 \text{ or } y\le 0}
 \gamma_{\delta_1}(\abs{y-x}) \left(u(y)-u(x)\right)^2dxdy\nonumber\\
&\; + \int_{x,y \in \Omega\cup \Omega_{\mathcal{I}}, x>0 \text{ and }  y>0} M \gamma_{\delta_1}(\abs{y-x}) \\
&\qquad\cdot\sum_{j=0}^{M-1} \left(u(x+\frac{j+1}{M}(y-x))-u(x+\frac{j}{M}(y-x))\right)^2\,dxdy.\nonumber
\end{align}
Comparing the difference between \eqref{1D_stab_eq1} and \eqref{nonlocal_bilinear}, {\helen{in order to obtain the statement \eqref{bilinear_QNL_vs_nonlocal}}}
 we only need show that
\begin{align}\label{1D_stab_eq2}
&\int_{x,y \in \Omega\cup \Omega_{\mathcal{I}}, x>0 \text{ and }  y>0}M \gamma_{\delta_1}(\abs{y-x})\\
&\qquad\qquad\quad\cdot\sum_{j=0}^{M-1} \left(u(x+\frac{j+1}{M}(y-x))-u(x+\frac{j}{M}(y-x))\right)^2dxdy\nonumber\\
&\quad \ge \int_{x,y \in \Omega\cup \Omega_{\mathcal{I}}, x>0 \text{ and }  y>0}\gamma_{\delta_1}(\abs{y-x})\left(u(y)-u(x)\right)^2dxdy.\nonumber
\end{align}
Since
\[
(a_1+\dots+a_{M})^2\le M(a_1^2+\dots+a_M^2),
\]
we have that
\begin{align*}
&\int_{x,y \in \Omega\cup \Omega_{\mathcal{I}}, x>0 \text{ and }  y>0}\gamma_{\delta_1}(\abs{y-x})\nonumber\\
& \quad \qquad\cdot M\sum_{j=0}^{M-1} \left(u(x+\frac{j+1}{M}(y-x))-u(x+\frac{j}{M}(y-x))\right)^2dxdy\nonumber\\
& \quad \ge \int_{x,y \in \Omega\cup \Omega_{\mathcal{I}}, x>0 \text{ and }  y>0}\gamma_{\delta_1}(\abs{y-x})\nonumber\\
&\qquad\quad
\cdot\left(\sum_{j=0}^{M-1} \left( u(x+\frac{j+1}{M}(y-x))-u(x+\frac{j}{M}(y-x)) \right)\right)^2dxdy\nonumber\\
&\quad = \int_{x,y \in \Omega\cup \Omega_{\mathcal{I}}, x>0 \text{ and }  y>0}\gamma_{\delta_1}(\abs{y-x})
\cdot\left( u(y)-u(x) \right)^2dxdy,
\end{align*}
which is exactly what we want in \eqref{1D_stab_eq2}. Therefore, we proved the proposition \eqref{bilinear_QNL_vs_nonlocal}.
\end{proof}

Since all nonlocal norm $\|\cdot \|_{\mathcal{S}_{\delta}}$ satisfies the Poincar\'e inequality \cite{Du2012a,Tian2016c}
\[
\|u\|_{{S}_{\delta}} \ge c \|u\|_{L^2(\Omega\cup\Omega_{\mathcal{I
    }})},
\]
as an immediate corollary of Proposition~\ref{lemma_L2_Nonlocal_vs_QNL}, we have the following $L^2$ stability
for the QNL coupling.
\begin{corollary}
The QNL coupling is $L^2$ stable:
\[
\qnlbilinear(u,u)\ge c \|u\|_{L^2 (\Omega\cup\Omega_{\mathcal{I }})}^2,\quad\forall u\in S_{\mathsf{qnl}}.
\]
\end{corollary}

Besides the $L^2$ stability (coercivity) of the bilinear form, we also
have the maximum principle (\textit{i.e.}, $L^{\infty}$ stability) of
the coupled diffusion since {the kernels are} non-negative.
\begin{proposition}[Maximum principle] \label{theorem_max_principle}
  {\helen{If $u\in C^1(\Omega)\cap \overline{C(\Omega\cup\Omega_{\mathcal{I}})}$ }}, then the following Dirichlet initial-boundary value problem with
  quasinonlocal diffusion
  \begin{equation*}
    \begin{cases}
      \displaystyle \frac{\partial u_{\mathsf{qnl}}}{\partial t}=\calL^{{\mathsf{qnl}}} u_{\mathsf{qnl}}(x)+f(x),\quad \forall x\in \Omega, \\
      u_{\mathsf{qnl}}(x,t)=g_d(x, t),\quad \forall x\in \Omega_{\mathcal{I}},\quad \forall {\helen{T\ge t\ge 0}}, \\
      u_{\mathsf{qnl}}(x,0)= u_{\mathsf{qnl}}^0(x), \quad \forall x\in
      \Omega
    \end{cases}
  \end{equation*}
  satisfies the maximum principle. That is
  \begin{equation*}
    u(x,t)\leq \max\left\{g_d(x, s)\Big|_{x\in\Omega_{\mathcal{I}}, \; 0 \leq s \leq t},\; u_{\mathsf{qnl}}^0(x)\Big|_{x\in \Omega}\right\}
  \end{equation*}
  if $f(x)\leq 0$ for $\forall\, x\in \Omega$, and similarly
  \begin{equation*}
    u(x,t )\geq \min\left\{g_d(x, s)\Big|_{x\in\Omega_{\mathcal{I}}, \; 0 \leq s \leq t},\; u_{\mathsf{qnl}}^0(x)\Big|_{x\in \Omega}\right\}
  \end{equation*}
  if $f(x)\geq 0$ for $\forall\, x\in \Omega$.
\end{proposition}
\helen{\begin{proof} Let us consider only the case $f(x)\le 0$ since
    the other case is similar.  We denote
    $Q_{T}:=\left(\Omega\cup \Omega_{\mathcal{I}}\right)\times (0,
    T)$.
    Fixing an arbitrary small positive number $\epsilon >0$, we define
    an auxiliary function $w:=u-\epsilon t$. We will first study $w$
    and then conclude information about $u$ by taking the limit
    $\epsilon \downarrow 0$.

Clearly, on $Q_{T}$, we have
\begin{equation}\label{w_eq1}
\begin{cases}
w\le u\le w+\epsilon T, \quad\text{on } Q_{T},\\
w_t-\calL^{{\mathsf{qnl}}} w(x)\le 0-\epsilon <0, \quad \forall x\in \Omega.
\end{cases}
\end{equation}
We claim that the maximum of $w$ on $Q_{T-\epsilon}$ occurs on $\partial_{p} Q_{T-\epsilon}:=\overline{\left(\Omega\cup\Omega_{\mathcal{I}}\right) } \times\{0\}  \cup  \Omega_{\mathcal{I}}\times (0, T-\epsilon]$.

Suppose the contrary, that is, $w(x, t)$ has its maximum at $(x^*,t^*)\in \overline{Q_{T-\epsilon}}$ with $0<t^*\le T-\epsilon$ and
$x^*\in \Omega=(-1,1)$.
Because $0<t^*\le T-\epsilon$, thus, $w_t(x^*,t^*)$ must be equal to $0$ if $0<t^* < T-\epsilon$,  and
$w_t(x^*,t^*)\ge 0$ if $t^* = T-\epsilon$. Meanwhile, because $w(x^*,t^*)$ is a maximum value and both diffusion kernels $\gamma_{\delta_1}$
 and $\gamma_{\delta_2}$ are non-negative, so from the definition of $\calL^{\mathsf{qnl}}$ \eqref{def_Lqnl}
we have
\[
\calL^{{\mathsf{qnl}}} w(x^*)\le 0.
\]
This immediately leads to $w_t(x^*,t^*)-\calL^{{\mathsf{qnl}}} w(x^*)\ge 0$, which contradicts the second expression of \eqref{w_eq1}. Therefore, the maximum of $w$ on $\overline{Q_{T-\epsilon}} $ occurs on $\partial_{p} Q_{T-\epsilon}$.

Now, we are going to prove that the maximum of $u$ occurs on $\partial_{p} Q_{T}$. Notice that
\begin{equation}\label{w_eq2}
\begin{cases}
&w\le u \text{ and }\partial_p Q_{T-\epsilon} \subset \partial_p Q_{T},\\
& \max\limits_{\overline{Q_{T-\epsilon}}} w=\max\limits_{\partial_p Q_{T-\epsilon}} w\le \max\limits_{ \partial_p Q_{T-\epsilon} } u\le \max\limits_{ \partial_p Q_{T} } u,
\end{cases}
\end{equation}
and $u\le w+\epsilon T$, so with \eqref{w_eq2}, we also have
\begin{equation}\label{w_eq3}
\max\limits_{\over{Q_{T-\epsilon}}} u\le \max\limits_{\over{Q_{T-\epsilon}}} w+\epsilon T\le \max\limits_{ \partial_p Q_{T} } u+\epsilon T.
\end{equation}
Because $u\in \overline{C(\Omega\cup\Omega_{\mathcal{I}})}$, we can obtain that
$
\max\limits_{\over{Q_{T-\epsilon}}} u \uparrow \max\limits_{\over{Q_{T}}} u  \text{ as } \epsilon\downarrow 0.
$
By allowing $\epsilon\downarrow 0$ and combining \eqref{w_eq3} together, we get
\[
\max\limits_{\overline{Q_T}} u=\lim_{\epsilon \downarrow 0} \max\limits_{\over{Q_{T-\epsilon}}} u\le \lim_{\epsilon \downarrow 0} \left(\max\limits_{ \partial_p Q_{T} } u+\epsilon T\right)
=\max\limits_{ \partial_p Q_{T} } u \le \max_{\overline{Q_T}} u.
\]
Therefore, we can conclude that $ \max\limits_{\overline{Q_T}} u =\max\limits_{ \partial_p Q_{T} } u$, which corresponds to the maximum principle when $f(x)\le 0$.
\end{proof}}

\section{Finite difference discretization}\label{section_fdm}

In this section, we discuss the discretization of the quasinonlocal
diffusion, for which the coupling is done at the continuous level. As
our main focus is to develop a consistent nonlocal coupling model
within a continuous framework, for the purpose of simplicity, we will {\helen{follow the idea in references \cite{Tian2013a,Tian2014a}}}
and just use a first order numerical scheme by a simple Riemann sum
approximation of the integral. Development of other type of
numerical discretization and higher order finite difference scheme will be left for
future works.

For concreteness, let us take the interval $\Omega=(-1, 1)$, which is
decomposed into $\Omega_1=(-1,0)$ and $\Omega_2=(0, 1)$ with interface
at $x=0$.  We divide the interval into $2N$ uniform subintervals
with equal length: $h=1/N$ and grid points $-1 = x_0 <x_1<\dots<x_{2N}=1$, so the interface grid point is $x_N = 0$.
The volume constrained region is
$\Omega_{\mathcal{I}}:={[-\delta_1-1,-1]\cup [1, 1+\delta_1]}$, where Dirichlet boundary condition $u = 0$ is assumed.

We assume that $\delta_1=r_1 \,h$, $\delta_2=r_2\, h$ with
$M:=r_1/r_2\in \mathbb{N}$. As we take a uniform mesh throughout the
domain, the stencil width is different in the two regions. We use the
scaling invariance of second moments of $\gamma_{\delta_1}$ and
$\gamma_{\delta_2}$ and approximate the quasinonlocal diffusion
operator $\calL^{\mathsf{qnl}}$ in the three regimes.

Case 1. For grid point $x_i\in [x_0, \; x_{N}=0]$, it corresponds to $\Omega_1$, thus,
\begin{equation*}
\begin{split}
\calL^{\mathsf{qnl}} u(x_i)
=&\int_{-\delta_1}^{\delta_1}
\left( u(x_i+s)-u(x_i)\right) \gamma_{\delta_1}(s)ds\\
=&\int_{0}^{\delta_1}
\left(\frac{ u(x_i+s)-2u(x_i)+u(x_i-s)}{s^2}\right) s^2\gamma_{\delta_1}(s)ds\\
=&\sum_{j=1}^{r_1}\left(\frac{ u(x_{i+j})-2u(x_i)+u(x_{i-j}) }{(jh)^2}\right)
\int_{(j-1)h}^{jh}s^2\gamma_{\delta_1}(s)ds + O(h).
\end{split}
\end{equation*}

Case 2.
For grid point $x_i \in [x_{N+r_1}=x_{N}+\delta_1,\;  x_{2N+1}]$, it corresponds to $\Omega_2$, thus,
\begin{equation*}
\begin{split}
\calL^{\mathsf{qnl}} u(x_i)
=&\int_{-\delta_2}^{\delta_2}
\left( u(x_i+s)-u(x_i)\right) \gamma_{\delta_2}(s)ds\\
=&\int_{0}^{\delta_2}
\left(\frac{ u(x_i+s)-2u(x_i)+u(x_i-s)}{s^2}\right) s^2\gamma_{\delta_2}(s)ds\\
=&\sum_{j=1}^{r_2}\left(\frac{ u(x_{i+j})-2u(x_i)+u(x_{i-j}) }{(jh)^2}\right)
\int_{(j-1)h}^{jh}s^2\gamma_{\delta_2}(s)ds + O(h).
\end{split}
\end{equation*}

Case 3.  For grid points
$ x_i \in [x_{N+1},\; x_{N+r_1-1}]\subset(0, \delta_1)$, this is the buffer
region, we have
\begin{equation*}
\begin{split}
\calL^{\mathsf{qnl}}& u(x_i)\\
=& \int_{-\delta_1}^{-x_i}\left(u(x_i+s)-u(x_i)\right) \gamma_{\delta_1}(s)ds
+\frac{1}{M}\int_{-\frac{1}{M}x_{i}}^{\delta_2}\left(u(x_i+s)-u(x_i)\right) \gamma_{\delta_2}(s)ds\\
&\quad + \frac{1}{M}\sum_{k=1}^{M-1} \int_{-\frac{1}{k}x_{i}}^{\frac{1}{k}x_{i}} \left(u(x_i+s)-u(x_i)\right) \gamma_{\delta_2}(s)ds\\
=& \int_{-\delta_1}^{-x_i}\frac{\left(u(x_i+s)-u(x_i)\right)}{s^2} s^2 \gamma_{\delta_1}(s)ds
+\frac{1}{M}\int_{-\frac{1}{M}x_{i}}^{\delta_2} \frac{\left(u(x_i+s)-u(x_i)\right)}{s^2} s^2 \gamma_{\delta_2}(s)ds\\
&\quad +\frac{1}{M}\sum_{k=1}^{M-1} \int_{0}^{\frac{1}{k}x_{i}}\frac{ \left(u(x_i+s)-2u(x_i)+u(x_i-s)\right)}{s^2}s^2 \gamma_{\delta_2}(s)ds\\
=:& T_1+T_2+T_3.
\end{split}
\end{equation*}
For the first term $T_1$, we approximate it by
\begin{equation*}
\begin{split}
T_1=& \int_{-\delta_1}^{-x_i}\frac{\left(u(x_i+s)-u(x_i)\right)}{s^2} s^2 \gamma_{\delta_1}(s)ds
=\int_{x_i}^{\delta_1}\frac{\left(u(x_i-s)-u(x_i)\right)}{s^2} s^2 \gamma_{\delta_1}(s)ds\\
=& \sum_{j=i-m+1}^{r_1}\left(\frac{ u(x_{i-j})-u(x_i)) }{(jh)^2}\right)
\int_{(j-1)h}^{jh}s^2\gamma_{\delta_1}(s)ds + O(h).
\end{split}
\end{equation*}
For the second term $T_2$, we have
\begin{equation}
\begin{split}
T_2&=\frac{1}{M}\int_{-\frac{1}{M}x_{i}}^{\delta_2} \frac{\left(u(x_i+s)-u(x_i)\right)}{s^2} s^2 \gamma_{\delta_2}(s)ds\\
&=\frac{1}{M}\int_{-\frac{1}{M}x_{i}}^{0} \frac{\left(u(x_i+s)-u(x_i)\right)}{s^2} s^2 \gamma_{\delta_2}(s)ds\\
&\qquad \qquad \quad +\frac{1}{M}\int_{0}^{\delta_2} \frac{\left(u(x_i+s)-u(x_i)\right)}{s^2} s^2 \gamma_{\delta_2}(s)ds\\
&=: T_{21}+T_{22}.
\end{split}
\end{equation}
For $T_{21}$:
\begin{equation*}
\begin{split}
T_{21}=&\frac{1}{M}\int_{-\frac{1}{M}x_{i}}^{0} \frac{\left(u(x_i+s)-u(x_i)\right)}{s^2} s^2 \gamma_{\delta_2}(s)ds\\
=&\frac{1}{M^2}\int_{-x_{i}}^{0} \frac{\left(u(x_i+\frac{s}{M})-u(x_i)\right)}{s^2} s^2 \gamma_{\delta_2}(\frac{s}{M})ds.
\end{split}
\end{equation*}
Note that $x_i - \frac{x_i}{M}$ corresponding to the left end of the above integration domain is not a grid point; we will use an interpolation for the value of $u$ at $x_i + \frac{s}{M}$ which leads to
\begin{equation*}
\left(u(x+\frac{s}{M})-u(x)\right)
\approx \frac{1}{M}\left(u(x+s)-u(x)\right),
\end{equation*}
hence, we approximate $T_{21}$ by
\begin{equation*}
\begin{split}
T_{21}
=&\frac{1}{M^2}\int_{-x_{i}}^{0} \frac{\left(u(x_i+\frac{s}{M})-u(x_i)\right)}{s^2} s^2 \gamma_{\delta_2}(\frac{s}{M})ds \\
=& \frac{1}{M^2}\int_{-x_{i}}^{0}\frac{1}{M} \frac{\left(u(x_i+s)-u(x_i)\right)}{s^2} s^2 \gamma_{\delta_2}(\frac{s}{M})ds + O(h)\\
=&\frac{1}{M^3}\sum_{j=1}^{i-m} \frac{\left(u(x_{i-j})-u(x_i)\right)}{(jh)^2}\int_{(j-1)h}^{jh} s^2 \gamma_{\delta_2}(\frac{s}{M})ds + O(h) \\
=&\sum_{j=1}^{i-m} \frac{\left(u(x_{i-j})-u(x_i)\right)}{(jh)^2}\int_{(j-1)h}^{jh} s^2 \gamma_{\delta_1}(s)ds + O(h).
\end{split}
\end{equation*}
For $T_{22}$,  it is
\begin{equation*}
\begin{split}
  T_{22}=&\frac{1}{M}\int_{0}^{\delta_2}
  \frac{\left(u(x_i+s)-u(x_i)\right)}{s^2} s^2 \gamma_{\delta_2}(s)ds \\
  = & \frac{1}{M^2}\int_{0}^{\delta_1} \frac{\left(u(x_i+\frac{s}{M})-u(x_i)\right)}{s^2} s^2 \gamma_{\delta_2}(\frac{s}{M})ds\\
  = & \frac{1}{M^3}\int_{0}^{\delta_1}
  \frac{\left(u(x_i+s)-u(x_i)\right)}{s^2} s^2
  \gamma_{\delta_2}(\frac{s}{M})ds + O(h) \\
  = & \sum_{j=1}^{r_1}
  \frac{\left(u(x_{i+j})-u(x_i)\right)}{(jh)^2}\int_{(j-1)h}^{(j)h}
  s^2 \gamma_{\delta_1}(s)ds + O(h).
\end{split}
\end{equation*}
For $T_{3}$, considering each $1\le k\le M-1$, we have two cases $\frac{1}{k}x_i<\delta_2$ and $\frac{1}{k}x_i\ge\delta_2$:
\begin{itemize}
\item If $\frac{1}{k}x_i<\delta_2$, we then handle $T_{3k}$ in a similar {way} to $T_{21}$ and get
\begin{equation*}
\begin{split}
T_{3k}:=& \frac{1}{M}\int_{0}^{\frac{1}{k}x_{i}}\frac{ \left(u(x_i+s)-2u(x_i)+u(x_i-s)\right)}{s^2}s^2 \gamma_{\delta_2}(s)ds\\
= &  \frac{1}{Mk^2}\int_{0}^{x_{i}}\frac{ \left(u(x_i+s)-2u(x_i)+u(x_i-s)\right)}{s^2}s^2 \gamma_{\delta_2}(\frac{s}{k})ds + O(h) \\
=&\frac{1}{Mk^2}\sum_{j=1}^{i-m}\frac{\left( u(x_{i+j})-2u(x_i)+u(x_{i-j}) \right)}{(jh)^2}\int_{(j-1)h}^{jh}s^2 \gamma_{\delta_2}(\frac{s}{k})ds + O(h).
\end{split}
\end{equation*}
\item If $\frac{1}{k}x_i\ge\delta_2$, then $T_{3k}$ is computed by
\begin{equation*}
\begin{split}
T_{3k}:=& \frac{1}{M}\int_{0}^{\frac{1}{k}x_{i}}\frac{ \left(u(x_i+s)-2u(x_i)+u(x_i-s)\right)}{s^2}s^2 \gamma_{\delta_2}(s)ds\\
=&\frac{1}{M}\int_{0}^{\delta_2}\frac{ \left(u(x_i+s)-2u(x_i)+u(x_i-s)\right)}{s^2}s^2 \gamma_{\delta_2}(s)ds\\
=&\frac{1}{M^2}\int_{0}^{\delta_2}\frac{ \left(u(x_i+\frac{s}{M})-2u(x_i)+u(x_i-\frac{s}{M})\right)}{s^2}s^2 \gamma_{\delta_2}(\frac{s}{M})ds\\
= & \int_{0}^{\delta_1}\frac{ \left(u(x_i+s)-2u(x_i)+u(x_i-s)\right)}{s^2}s^2 \gamma_{\delta_1}(s)ds + O(h)\\
=&\sum_{j=1}^{r_1}\frac{ \left(u(x_{i+j})-2u(x_i)+u(x_{i-j})\right)}{(jh)^2} \int_{(j-1)h}^{(j)h}s^2 \gamma_{\delta_1}(s)ds + O(h).
\end{split}
\end{equation*}
\end{itemize}

It is straightforward to check that the resulting finite difference
approximation to the diffusion operator preserves the symmetry and is
also positive definite.

\section{Numerical results}\label{section_num}
In this section, we will consider serval benchmark problems to check the accuracy and stability performance of the numerical scheme.
The expression of $\gamma_{\delta}(s)$ is fixed to be
\[
\gamma_{\delta}(s)=\frac{2}{\delta^2 s}.
\]
The time discretization is just the simple Euler method with
$\Delta t=\kappa_{\textsf{cfl}} h^2$, $\kappa_{\textsf{cfl}}$ is set
to be $1/4$. The patch-test consistency, symmetry and positive
definiteness of the finite difference matrix are validated numerically.


We first consider the following one-dimensional volume-constrained
Dirichlet \; problem
\[
u(x,0)=x^2\, (1-x^2),
\quad
f(x)=e^{-t}(12 x^2-2)-e^{-t}x^2(1-x^2).
\]
The corresponding limiting local diffusion problem as $\delta \rightarrow 0$ is
\begin{equation}\label{Local_diff_examp1}
\begin{cases}
\frac{\partial u}{\partial t}- u_{xx}=f(x), & -1<x<1, \; \forall\, t>0, \\
u(x,0)=x^2\, (1-x^2), &  -1<x<1,\\
u(-1,t)=u(1,t)\equiv 0, & \forall\, t>0.
\end{cases}
\end{equation}
The exact solution for the limiting local diffusion problem is
\[
u_{\text{exact,local}}=e^{-t}\,x^2\,(1-x^2).
\]
We consider three cases: Case A: $\delta_1=6h$ and $\delta_2=2h$, with
$M=3$; Case B: $\delta_1=3h$ and $\delta_2=h$, and Case C:
$\delta_1=4h$ and $\delta_2=2h$, with $M=2$.  Note that in Case B, the
numerical scheme is effectively a coupling of {local kernel} with a
three-point stencil, and thus can be viewed as a nonlocal-to-local
coupling (on the level of numerical discretization).  The final
simulation time is $T=1$. Note that as $h \to 0$, not only that we
refine the mesh, but also the nonlocal diffusion model converges to
the local one. This numerical test thus verifies both convergence
(i.e., both the discretization error and modeling discrepancy go to
$0$).  We compute the $L^{\infty}$ difference between the quasinonlocal
solutions and the limiting local solution. The results are listed in
Table~\ref{table_limit_diff}. We observe the first order convergence
rate due to the numerical discretization of the quasinonlocal
diffusion in all three cases.

\begin{table}[h]
\begin{center}
\begin{tabular}{|c|c|c|c|c|c|c|}
\hline
{$h$} & {Case A} & Order &{ Case B} & Order &{Case C} & Order\\
\hline
$1/50$ & $ 6.132e$-$2$ & -  & $2.506e$-$2$   & -& $3.720e$-$2 $& - \\
\hline
$1/100$ & $ 3.018e$-$2$ & $1.02$ & $1.259e$-$2$& $0.99$  & $1.856e$-$2$ & $1.00$\\
\hline
$1/200$ & $ 1.506e$-$2$ & $1.00$ & $6.340e$-$3$& $0.99$ & $9.320e$-$3$& $0.99$\\
\hline
$1/400$ &$7.556e$-$3$ & $1.00$ & $3.192e$-$3 $ & $0.99 $& $4.687e$-$3$ & $1.00$\\
\hline
\end{tabular}
\smallskip
\caption{$L^{\infty}$ difference (diff) of \eqref{Local_diff_examp1} Case A :
  $\delta_1=6h$ and $\delta_2=2h$; Case B: $\delta_1=3h$ and
  $\delta_2=h$; Case C: $\delta_1=4h$ and $\delta_2=2h$.  The final
  simulation time is $T=1$.}\label{table_limit_diff}
\end{center}
\end{table}

{\helen{ We also computed the errors measured in the energy norm, which is defined as
\begin{equation}\label{def_error_energy}
\text{Energy err}:= \max\limits_{0\le t\le T} \|\nabla u(x, t)-\nabla u_{\text{exact,local}}(x, t)\|_{L^2(\Omega\cup\Omega_{\mathcal{I}})}.
\end{equation}
\helen{The discrete gradients are approximated by second order central
  finite difference.}  The results are listed in
Table~\ref{table_energy_diff}. We observe that the convergence order
is just around $0.5$ rather than $1$. 
This is due to the artificial boundary layer of
$\nabla u$ because the local limiting solution
$u_{\text{exact,local}}$ is not equal to zero on
$\Omega_{\mathcal{I}}=[-1-\delta_1,-1]\cup [1, 1+\delta_1]$ (see
Figure~\ref{Figure:du_boundary} for demonstration).
\begin{table}[h]
\begin{center}
\begin{tabular}{|c|c|c|c|c|}
\hline
{$h$} & {$\text{Energy err}$ of Case A} &{ Order} & {$\text{Energy err}$ of Case B} & { Order}\\
\hline
$1/50$ & $2.820e$-$1$ & $-$ & $2.679e$-$1$ & $-$ \\
\hline
$1/100$ & $2.065e$-$1$ & $0.45$ & $1.983e$-$1$ & $0.43$\\
\hline
$1/200$ & $ 1.486e$-$1$ & $0.47$ & $1.434e$-$1$ & $0.47$\\
\hline
$1/400$ &$1.060e$-$1$& $ 0.49$ & $1.025e$-$1$ & $0.49$\\
\hline
\end{tabular}
\smallskip
\caption{ Errors of quasinonlocal solution and local limiting solution \eqref{Local_diff_examp1} measured in the energy norm \eqref{def_error_energy}. Case A :
  $\delta_1=6h$ and $\delta_2=2h$; Case B: $\delta_1=4h$ and $\delta_2=2h$.  The final
  simulation time is $T=1$.}\label{table_energy_diff}
\end{center}
\end{table}
\begin{figure}[htp!]
\centering
\subfigure[$h=1/50$]{
\includegraphics[height =4 cm, width=0.3\textwidth]{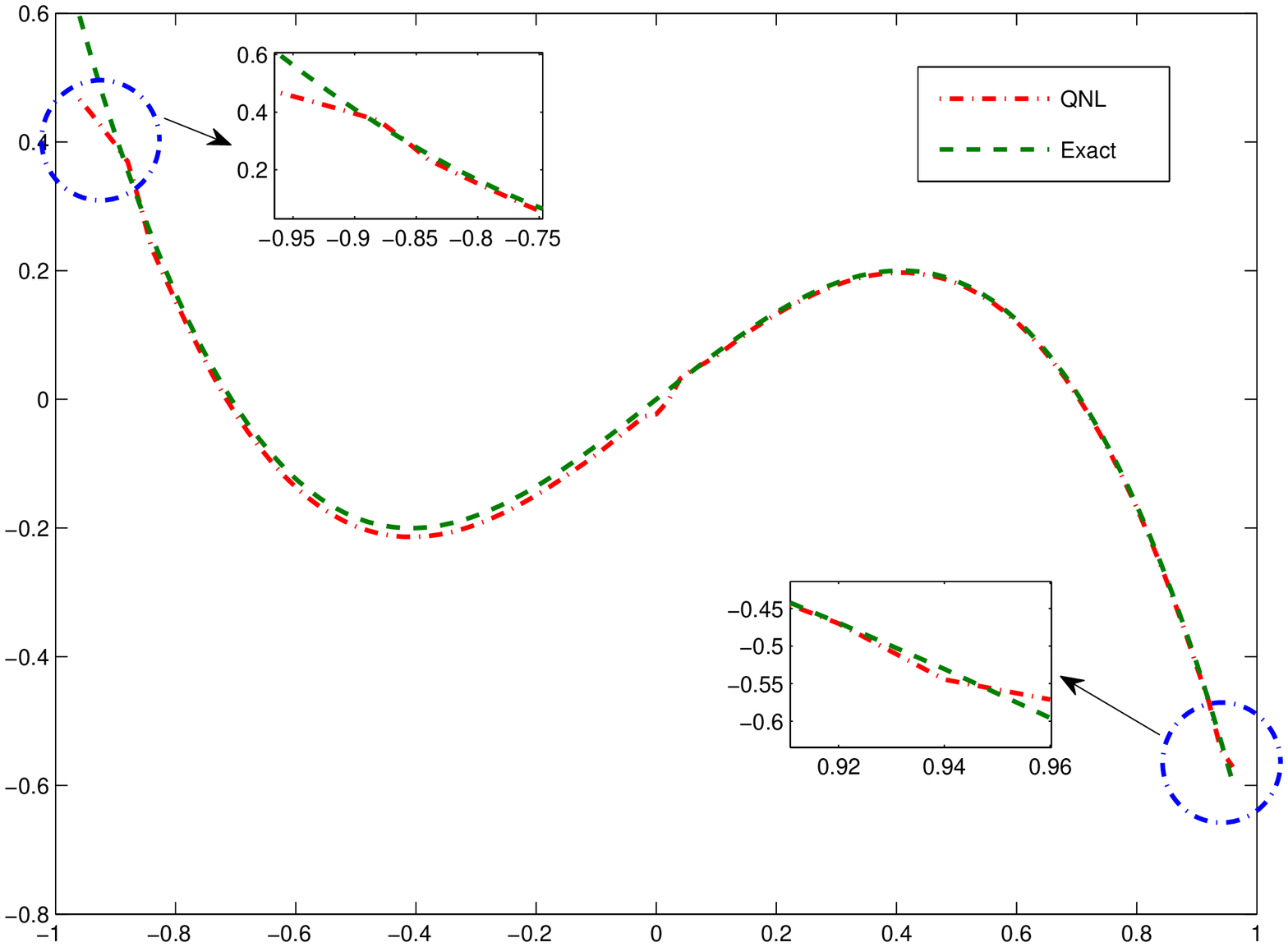}}
\subfigure[$h=1/100$]{
\includegraphics[height =4 cm, width=0.3\textwidth]{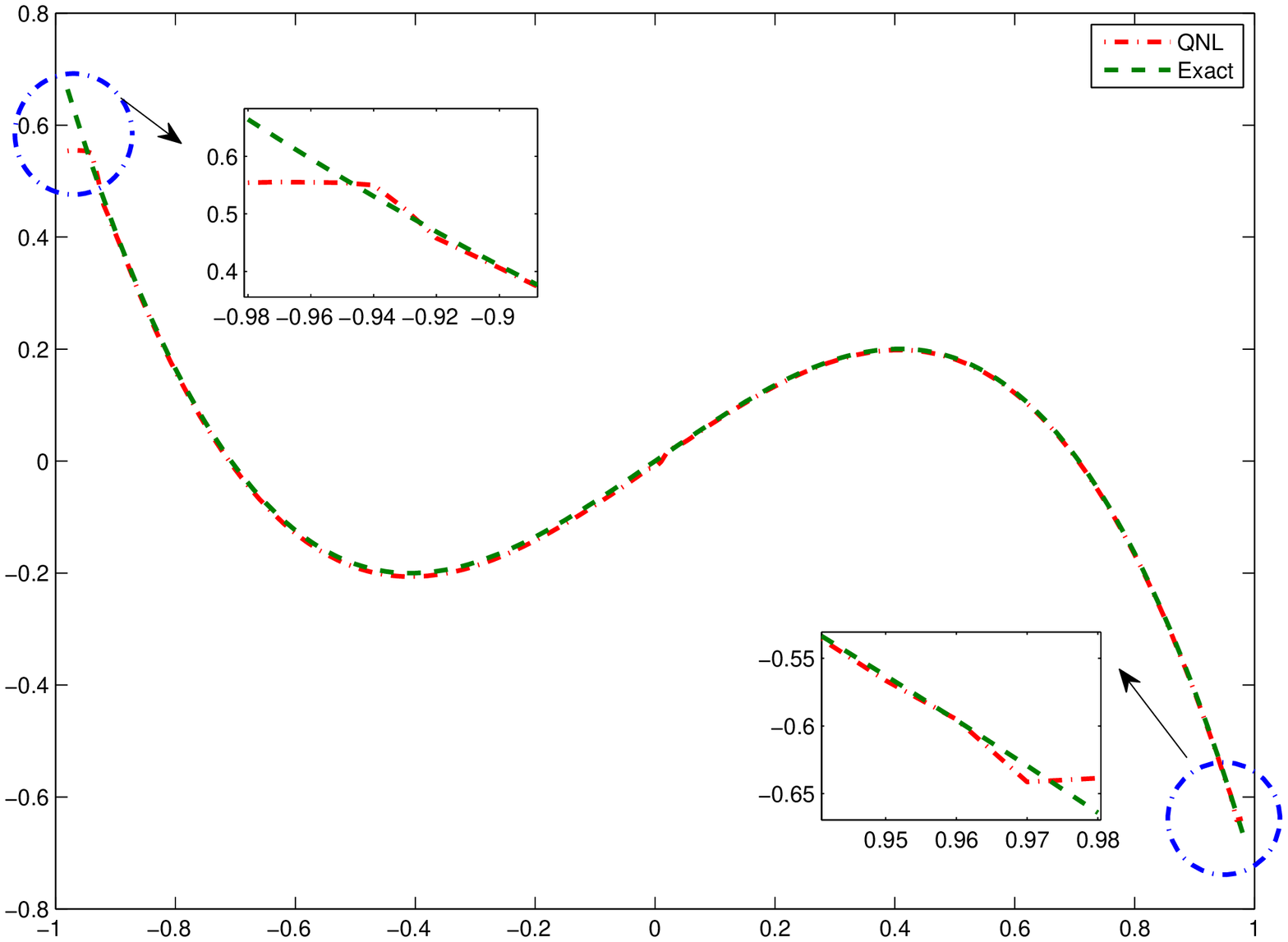}}
\subfigure[$h=1/200$]{
\includegraphics[height =4 cm, width=0.3 \textwidth]{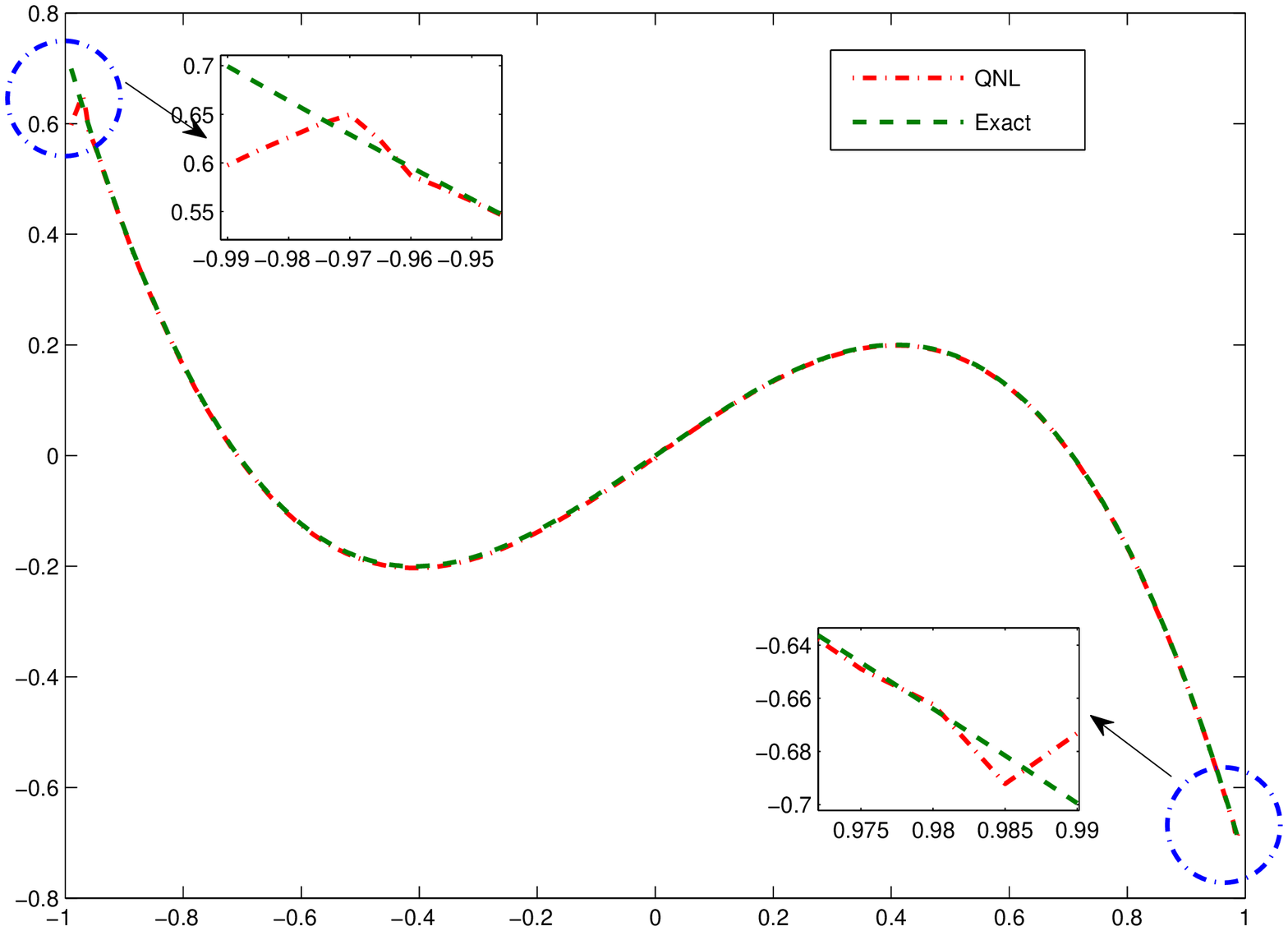}}
\vspace{- 10 pt}
\caption{Plots of displacement gradients (strains) for the
  quasinonlocal diffusion with $\delta_1=6h$ and $\delta_2=2h$
  versus the local limiting diffusion at $T=1$ with
  various $h$. The sizes of boundary layers are of $O(2\delta_1)$ and $O(2\delta_2)$ on both sides, respectively. }\label{Figure:du_boundary}
\end{figure}
To further study the origin of the loss of convergence order, we
compute the errors of quasinonlocal solution and local limiting
solution \eqref{Local_diff_examp1} measured in the energy norm within the
interior of $\Omega$, that is, the errors are only measured within
$[-1/2,\; 1/2]$ that contains the interface $x_0=0$:
\begin{equation}\label{def_error_energy2}
\text{Int energy err}:= \max\limits_{ 0\le t\le T} \|\nabla u(x, t)-\nabla u_{\text{exact,local}}(x, t)\|_{L^2([-1/2,\;1/2])}.
\end{equation}
This time, we clearly observe the first order convergence rate in
Table~\ref{table_interior_energy_diff}, further confirming that the
loss of convergence is from boundary layer. In fact, this is a known
problem for numerically imposing volumetric boundary condition (see
e.g., \cite{Tian2013a}), which we will not go into further details here.
\begin{table}[h]
\begin{center}
\begin{tabular}{|c|c|c|c|c|}
\hline
{$h$} & {Int $\text{energy err}$ of Case A} &{ Order} & {Int $\text{energy err}$ of Case B} & { Order}\\
\hline
$1/50$ & $2.920e$-$2$ & $-$ & $1.779$-$2$ & $-$ \\
\hline
$1/100$ & $1.383e$-$2$ & $1.07$ & $8.629e$-$3$ & $1.04$\\
\hline
$1/200$ & $6.716e$-$3$ & $1.04$ & $4.247e$-$3$ & $1.02$\\
\hline
$1/400$ & $3.063e$-$3$& $ 1.13$ & $2.106e$-$3$ & $1.01$\\
\hline
\end{tabular}
\smallskip
\caption{ Interior errors of quasinonlocal solution and local limiting solution \eqref{Local_diff_examp1} measured in energy norm \eqref{def_error_energy2}. Case A: $\delta_1=6h$ and $\delta_2=2h$; Case B: $\delta_1=4h$ and $\delta_2=2h$.  The final
  simulation time is $T=1$.}\label{table_interior_energy_diff}
\end{center}
\end{table}
}
}

Next we fix $h=1/200$, $\delta_1=5h$, $\delta_2=h$, and consider
initial datum which has a singularity at $x^*=-0.45+h/2$.
\[
u(x,0)=\frac{\sin(\pi x)}{x-x^*},
\quad
f(x)=0.
\]
The solution $u(x,t)$ is plotted for $T=1/4$ in Figure~\ref{Figure:
  defect_external_force}. We can see that the quasinonlocal diffusion
matches the fully nonlocal model, whereas the result of the fully
local diffusion is distinguishable from that of fully nonlocal model.
\begin{figure}[htp!]
\centering
\includegraphics[height =5.5 cm]{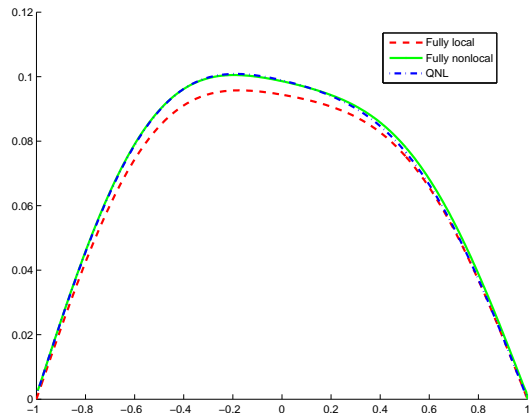}
\caption{The solution $u(x,t)$ is plotted for $T=1/4$.}\label{Figure: defect_external_force}
\end{figure}

\section{Conclusion}\label{section_con}
We have {\helen{proposed a new {\helen{self-adjoint}}, consistent and stable coupling
strategy for nonlocal diffusion problems in one dimensional space, which couples two nonlocal operators associated with different
horizon parameters $\delta_1$ and $\delta_2$ with $M:=\frac{\delta_1}{\delta_2}$ being an integer}}. This new coupling model is proved to be
{self-adjoint} and patch-test consistent.  In addition, the quasinonlocal
diffusion is also stable (coercive) with respect to the energy norm
induced by the nonlocal diffusion kernels as well as the $L^2$ norm,
and it satisfies the maximum principle.

We also consider a first order finite difference approximation to
discretize the continuous coupling model. This numerical approximation
preserves the {self-adjointness}, consistency, coercivity and the maximum
principle. The numerical scheme is validated through several examples.

Also, as for future works, {\helen{another}} immediate direction is extending the
coupling scheme to higher dimensions; as already mentioned, since the
nonlocal diffusion model only involves pairwise interactions in the
form of $u(y) - u(x)$, the extension should not pose serious
difficulties. Better numerical approximation to the continuous
quasinonlocal diffusion operator is also another interesting direction
to pursue. Another interesting topic is to couple the nonlocal diffusion operator directly with local diffusion (Laplace) operator in the framework of the quasinonlocal coupling (see \cite{Bobaru2008,Han2012,Delia2015a,Seleson2015a,Tian2016a} for some recent works that {couple the local and nonlocal diffusions together}).

\appendix
\section{Derivation of $\mathcal{L}^{\mathsf{qnl}}$}\label{section_app}
We give the deviation of the coupled diffusion operator
$\calL^{{\mathsf{qnl}}}$ \eqref{def_Lqnl}, stated as the following
Proposition. The calculation is straightforward but somewhat tedious.

\begin{proposition}\label{thm_qnl_diffusion}
  The coupled quasinonlocal energy functional $E^{{\mathsf{tot}},{\mathsf{qnl}}}$ induces the quasinonlocal diffusion
  operator $\calL^{\mathsf{qnl}}$ defined in \eqref{def_Lqnl}.
\end{proposition}
\begin{proof}
The first variation of $E^{{\mathsf{tot}},{\mathsf{qnl} }}$ with test function $\forall v\in S_{\mathsf{qnl}}$ is
\begin{align*}
\la& \partial E^{{\mathsf{tot}},{\mathsf{qnl}}}(u), v\ra\nonumber\\
=& \frac{1}{2}\int_{x,y \in \mathbb{R}, x\le 0 \text{ or } y\le 0}
 \gamma_{\delta_1}(\abs{y-x})dxdy \left(u(y)-u(x)\right)\left(v(y)-v(x)\right)\nonumber\\
&+\frac{1}{2}\int_{x,y \in \mathbb{R}, x>0 \text{ and }  y>0} dxdy \Big[\gamma_{\delta_1}(\abs{y-x}) \nonumber\\
&\qquad\cdot \frac{1}{M}\sum_{j=0}^{M-1} \left(u(x+\frac{j+1}{M}(y-x))-u(x+\frac{j}{M}(y-x))\right)\nonumber\\
&\qquad\qquad \qquad \left(v(x+\frac{j+1}{M}(y-x))-v(x+\frac{j}{M}(y-x))\right)M^2\Big]
=:T_1+T_2.
\end{align*}
Because of the symmetry in the $T_1$ integral in $x$ and $y$, we can convert $T_1$ to
\begin{equation}\label{qnl_first_var_T1}
\begin{split}
T_1=&\frac{1}{2}\int_{x,y \in \mathbb{R}, x\le 0 \text{ or } y\le 0}
 \gamma_{\delta_1}(\abs{y-x})dxdy \left(u(y)-u(x)\right)\left(v(y)-v(x)\right)\\
=&\int_{x,y \in \mathbb{R}, x\le 0 \text{ or } y\le 0}
 \gamma_{\delta_1}(\abs{y-x})dxdy \left(u(x)-u(y)\right)\cdot v(x).
\end{split}
\end{equation}
We now focus on $T_2$.
\begin{align}\label{qnl_first_var_T2_eq1}
T_2=&\frac{1}{2}\sum_{j=0}^{M-1} \int_{x,y \in \mathbb{R}, x>0 \text{ and }  y>0} dxdy\Big[M \gamma_{\delta_1}(\abs{y-x}) \\
&\qquad\cdot \left(u(x+\frac{j+1}{M}(y-x))-u(x+\frac{j}{M}(y-x))\right)\nonumber\\
&\qquad\qquad\left(v(x+\frac{j+1}{M}(y-x))-v(x+\frac{j}{M}(y-x))\right)\Big]\nonumber\\
=& \frac{1}{2}\sum_{j=0}^{M-1}
\int_{x,y \in \mathbb{R}, x>0 \text{ and }  y>0} dxdy\Big[\frac{1}{M^2} \gamma_{\delta_2}\left(\frac{\abs{y-x}}{M}\right) \nonumber\\
&\qquad \cdot\left(u(x+\frac{j+1}{M}(y-x))-u(x+\frac{j}{M}(y-x))\right)\nonumber\\
&\qquad\qquad \quad \left(v(x+\frac{j+1}{M}(y-x))-v(x+\frac{j}{M}(y-x))\right)\Big]\nonumber\\
=&\frac{1}{2}\sum_{j=0}^{M-1}
\int_{x,y \in \mathbb{R}, x>0 \text{ and }  y>0} dxdy\Big[\frac{1}{M^2} \gamma_{\delta_2}\left(\frac{\abs{y-x}}{M}\right) \nonumber\\
&\qquad\qquad \cdot\left(u(x+\frac{j+1}{M}(y-x))-u(x+\frac{j}{M}(y-x))\right) v\left(x+\frac{j+1}{M}(y-x)\right)\Big]\nonumber\\
&-\frac{1}{2}\sum_{j=0}^{M-1}
\int_{x,y \in \mathbb{R}, x>0 \text{ and }  y>0} dxdy\Big[\frac{1}{M^2} \gamma_{\delta_2}\left(\frac{\abs{y-x}}{M}\right) \nonumber\\
&\qquad\qquad\quad \cdot\left(u(x+\frac{j+1}{M}(y-x))-u(x+\frac{j}{M}(y-x))\right) v\left(x+\frac{j}{M}(y-x)\right)\Big].\nonumber
\end{align}
Let $k:=(M-1)-j=M-(j+1)$ in the second summation term of \eqref{qnl_first_var_T2_eq1}, we get
\begin{align}\label{qnl_first_var_T2_eq2}
&T_2= \frac{1}{2}\sum_{j=0}^{M-1}
\int_{x,y \in \mathbb{R}, x>0 \text{ and }  y>0} dxdy\Big[\frac{1}{M^2} \gamma_{\delta_2}\left(\frac{\abs{y-x}}{M}\right) \\
&\qquad \cdot\left(u(x+\frac{j+1}{M}(y-x))-u(x+\frac{j}{M}(y-x))\right) v\left(x+\frac{j+1}{M}(y-x)\right)\Big]\nonumber\\
&-\frac{1}{2}\sum_{k=0}^{M-1}
\int_{x,y \in \mathbb{R}, x>0 \text{ and }  y>0} dxdy\Big[\frac{1}{M^2} \gamma_{\delta_2}\left(\frac{\abs{y-x}}{M}\right) \nonumber\\
&\cdot\left(u(\frac{k}{M}x+(1-\frac{k}{M})y)-u(\frac{k+1}{M}x+(1-\frac{k+1}{M})y)\right) v\left(\frac{k+1}{M}x+(1-\frac{k+1}{M})y\right)\Big]\nonumber\\
&= \frac{1}{2}\sum_{j=0}^{M-1}
\int_{x,y \in \mathbb{R}, x>0 \text{ and }  y>0} dxdy\Big[\frac{1}{M^2} \gamma_{\delta_2}\left(\frac{\abs{y-x}}{M}\right) \nonumber\\
&\qquad \cdot\left(u(x+\frac{j+1}{M}(y-x))-u(x+\frac{j}{M}(y-x))\right) v\left(x+\frac{j+1}{M}(y-x)\right)\Big]\nonumber\\
&+\frac{1}{2}\sum_{k=0}^{M-1}
\int_{x,y \in \mathbb{R}, x>0 \text{ and }  y>0} dxdy\Big[\frac{1}{M^2} \gamma_{\delta_2}\left(\frac{\abs{y-x}}{M}\right) \nonumber\\
&\qquad\quad\cdot\left(u(y+\frac{k+1}{M}(x-y))-u(y+\frac{k}{M}(x-y))\right)
v\left(y+\frac{k+1}{M}(x-y)\right)\Big].\nonumber
\end{align}
Changing the notation order of $x$ and $y$ in the second integral of \eqref{qnl_first_var_T2_eq2}, we have
\begin{align}\label{qnl_first_var_T2_eq3}
T_2=&\sum_{j=0}^{M-1}
\int_{x,y \in \mathbb{R}, x>0 \text{ and }  y>0} dxdy\Big[\frac{1}{M^2} \gamma_{\delta_2}\left(\frac{\abs{y-x}}{M}\right) \nonumber\\
&\qquad \cdot\left(u(x+\frac{j+1}{M}(y-x))-u(x+\frac{j}{M}(y-x))\right) v\left(x+\frac{j+1}{M}(y-x)\right)\Big].
\end{align}
Now let
$ z:=x+\frac{j+1}{M}(y-x)$ to replace $y$, then the integration interval for $z$ becomes
\[
z=\left(1-\frac{j+1}{M}\right)x+\frac{j+1}{M}y >\left(1-\frac{j+1}{M}\right)x.
\]
\eqref{qnl_first_var_T2_eq3} thus becomes
\begin{align}\label{qnl_first_var_T2_eq4}
T_2=& \sum_{j=0}^{M-1}
\int_{x>0 \text{ and }  z>(1-\frac{j+1}{M})x} \Big[\frac{1}{M(j+1)} \\
&\qquad \gamma_{\delta_2}\left(\frac{\abs{z-x}}{j+1}\right)
\cdot\left(u(z)-u(z+\frac{1}{j+1}(x-z))\right) v\left(z\right)\Big]~dxdz\nonumber\\
=&\frac{1}{M}\sum_{j=0}^{M-1}
\int_{z>0 \text{ and }  0<x<\frac{M}{M-(j+1)}z} dxdz\Big[\frac{1}{(j+1)} \gamma_{\delta_2}\left(\frac{\abs{z-x}}{j+1}\right) \nonumber\\
&\qquad \qquad \qquad \qquad\qquad\qquad \cdot\left(u(z)-u(z+\frac{1}{j+1}(x-z))\right) v\left(z\right)\Big],\nonumber
\end{align}
where $\frac{M}{M-(j+1)}z$ is formally regarded as $+\infty$ when $j=M-1$.

Now let $w=z+\frac{1}{j+1}(x-z)=\frac{j}{j+1}z+\frac{1}{j+1}x$ to replace x, thus the integration interval for $w$ is
\[
z-\frac{1}{j+1}z<w<z+\frac{1}{M-(j+1)}z,
\]
and we have $T_2$:
\begin{align}\label{qnl_first_var_T2_eq5}
T_2=&\frac{1}{M} \sum_{j=0}^{M-1}
\int_{z>0 \text{ and }  z-\frac{1}{j+1}z<w<z+\frac{1}{M-(j+1)}z} \Big[\gamma_{\delta_2}\left(\abs{w-z}\right)
\cdot\left(u(z)-u(w)\right) v\left(z\right)\Big]~dwdz\nonumber\\
=&\frac{1}{M} \sum_{j=1}^{M}
\int_{z>0 \text{ and }  z-\frac{1}{j}z<w<z+\frac{1}{M-(j)}z} \Big[\gamma_{\delta_2}\left(\abs{w-z}\right)
\cdot\left(u(z)-u(w)\right) v\left(z\right)\Big]~dwdz\nonumber\\
=&\frac{1}{M}\sum_{j=1}^{M-1}
\int_{z>0 \text{ and }  z-\frac{1}{j}z<w<z+\frac{1}{j}z} \Big[ \gamma_{\delta_2}\left(\abs{w-z}\right)
\cdot\left(u(z)-u(w)\right) v\left(z\right)\Big]~dwdz\nonumber\\
&\quad +\frac{1}{M}\int_{z>0 \text{ and }  z-\frac{1}{M}z<w<\infty} \Big[ \gamma_{\delta_2}\left(\abs{w-z}\right)
\cdot\left(u(z)-u(w)\right) v\left(z\right)\Big]~dwdz.
\end{align}
Now, we combine $T_1$ \eqref{qnl_first_var_T1} and $T_2$ \eqref{qnl_first_var_T2_eq5} together, we have
\begin{align}\label{qnl_first_variation_eq2}
\la& \partial E^{{\mathsf{tot}},{\mathsf{qnl}}}(u), v\ra\nonumber\\
&=\int_{x,y \in \mathbb{R}, x\le 0 \text{ or } y\le 0}
 \gamma_{\delta_1}(\abs{y-x})  \left(u(x)-u(y)\right)\cdot v(x) dxdy\nonumber\\
&\quad+
\frac{1}{M}\sum_{j=1}^{M-1}
\int_{z>0 \text{ and }  z-\frac{1}{j}z<w<z+\frac{1}{j}z} \Big[ \gamma_{\delta_2}\left(\abs{w-z}\right)
\cdot\left(u(z)-u(w)\right) v\left(z\right)\Big]~dwdz\nonumber\\
&\qquad +\frac{1}{M}\int_{z>0 \text{ and }  z-\frac{1}{M}z<w<\infty} \Big[ \gamma_{\delta_2}\left(\abs{w-z}\right)
\cdot\left(u(z)-u(w)\right) v\left(z\right)\Big]~dwdz\nonumber\\
&=\int_{x,y \in \mathbb{R}, x\le 0 \text{ or } y\le 0}
 \gamma_{\delta_1}(\abs{y-x}) \left(u(x)-u(y)\right)\cdot v(x) dxdy\nonumber\\
&\quad+
\frac{1}{M} \sum_{j=1}^{M-1}
\int_{x>0 \text{ and }  x-\frac{1}{j}x<y<x+\frac{1}{j}x} \Big[\gamma_{\delta_2}\left(\abs{y-x}\right)
\cdot\left(u(x)-u(y)\right) v\left(x\right)\Big]~dxdy\nonumber\\
&\qquad +\frac{1}{M}\int_{x>0 \text{ and }  x-\frac{1}{M}x<y<\infty} \Big[ \gamma_{\delta_2}\left(\abs{y-x}\right)
\cdot\left(u(x)-u(y)\right) v\left(x\right)\Big]~dxdy,
\end{align}
where we just replaces the notations $z$ by $x$ and $w$ by $y$.

The corresponding diffusion operator $\calL^{\mathsf{qnl}}$ is equal to negative of the first order variation of $ E^{{\mathsf{tot}},{\mathsf{qnl}}}(u)$,
which can be discussed in three cases below:
\begin{enumerate}
\item{Case I: $x\le 0$:
\[
\calL^{\mathsf{qnl}}u(x)=\int_{y\in \RR}
 \gamma_{\delta_1}(\abs{y-x}) \left(u(x)-u(y)\right)dy.
\]
}
\item{Case II: $0< x < \delta_1$:
\begin{equation*}
\begin{split}
\calL^{\mathsf{qnl}}u(x)=&
\int_{x-\delta_1<y<0}\gamma_{\delta_1}(\abs{y-x}) \left(u(x)-u(y)\right)dy\\
&\quad + \frac{1}{M}\sum_{j=1}^{M-1}
\int_{x-\frac{1}{j}x<y<x+\frac{1}{j}x} \gamma_{\delta_2}\left(\abs{y-x}\right)
\cdot\left(u(x)-u(y)\right) dy\nonumber\\
&\qquad\qquad + \frac{1}{M}\int_{x-\frac{1}{M}x<y<\infty} \gamma_{\delta_2}\left(\abs{y-x}\right)
\cdot\left(u(x)-u(y)\right)dy.
\end{split}
\end{equation*}
}
\item{Case III: $x\ge \delta_1$:
\begin{equation*}
\begin{split}
\calL^{\mathsf{qnl}}u(x)=&
 \frac{1}{M}\sum_{j=1}^{M-1}
\int_{x-\frac{1}{j}x<y<x+\frac{1}{j}x} \gamma_{\delta_2}\left(\abs{y-x}\right)
\cdot\left(u(x)-u(y)\right) dy\nonumber\\
&\qquad\qquad +\frac{1}{M} \int_{x-\frac{1}{M}x<y<\infty} \gamma_{\delta_2}\left(\abs{y-x}\right)
\cdot\left(u(x)-u(y)\right)dy.
\end{split}
\end{equation*}
Notice that $x\ge \delta_1$, thus
\[
x-\frac{1}{j}x<x-\frac{1}{M}x<x-\frac{1}{M}\delta_1=x-\delta_2,
\]
and
\[
x+\frac{1}{j}x>x+\frac{1}{M}x>x+\frac{1}{M}\delta_1=x+\delta_2.
\]
 Because outside the support, the diffusion kernel is zero, therefore, we have
\begin{equation*}
\begin{split}
\calL^{\mathsf{qnl}}u(x)=&
\frac{1}{M}\sum_{j=1}^{M-1}
\int_{x-\delta_2<y<x+\delta_2}  \gamma_{\delta_2}\left(\abs{y-x}\right)
\cdot\left(u(x)-u(y)\right) dy\nonumber\\
&\qquad\qquad +\frac{1}{M}\int_{x-\delta_2<y<\infty}  \gamma_{\delta_2}\left(\abs{y-x}\right)
\cdot\left(u(x)-u(y)\right)dy\\
=&\int_{x-\delta_2<y<x+\delta_2}\gamma_{\delta_2}\left(\abs{y-x}\right)
\cdot\left(u(x)-u(y)\right) dy\\
=&\int_{y\in \RR}\gamma_{\delta_2}\left(\abs{y-x}\right)
\cdot\left(u(x)-u(y)\right) dy.
\end{split}
\end{equation*}
}
\end{enumerate}
Hence, we get \eqref{def_Lqnl}.
\end{proof}


\bibliographystyle{abbrv}
\bibliography{QNLcouple}

\end{document}